\documentclass{amsart}
\usepackage{amssymb,latexsym}
\usepackage{newlattice}

\renewcommand{\bysame}{\makebox[3em]{\hrulefill}\thinspace}

\DeclareMathOperator{\Princl}{Princ}
\DeclareMathOperator{\Hom}{Hom}

\DeclareMathOperator{\Base}{Base}

\DeclareMathOperator{\Top}{Top}
\DeclareMathOperator{\Btm}{Btm}
\DeclareMathOperator{\Lat}{Lat}
\DeclareMathOperator{\Equi}{Equi}
\DeclareMathOperator{\Frame}{Frame}
\DeclareMathOperator{\Ordc}{{\E P}^3}

\newcommand{\zo}{$\set{0,1}$-}
\newcommand{\zs}{$0$-separating\xspace}

\newtheorem{theorem}{Theorem}
 
\newtheorem{lemma}[theorem]{Lemma}

\theoremstyle{definition}

\newtheorem{problem}{Problem}

\begin{document}
\title[Homomorphisms and principal congruences]{Homomorphisms and principal congruences \\ of bounded lattices}  
\author{G. Gr\"{a}tzer} 
\email[G. Gr\"atzer]{gratzer@me.com}
\address{Department of Mathematics\\
  University of Manitoba\\
  Winnipeg, MB R3T 2N2\\
  Canada}
\urladdr[G. Gr\"atzer]{http://server.maths.umanitoba.ca/homepages/gratzer/}
\date{August 15, 2015}
\subjclass[2010]{Primary: 06B10.}
\keywords{bounded lattice, congruence, principal, order.}

\begin{abstract}
Two years ago, I characterized the order $\Princl L$ of principal congruences 
of a bounded lattice $L$ as a bounded order.

If $K$ and $L$ are bounded lattices and $\gf$ is a 
\zo homomorphism of $K$ into~$L$, then there is a natural 
isotone \zo-map $\gf_{\Hom}$ from  $\Princl K$ into $\Princl L$.

We prove the converse: For bounded orders $P$ and $Q$
and an isotone \zo map $\gy$ of $P$ into $Q$,
we represent $P$ and $Q$ as $\Princl K$ and $\Princl L$ 
for bounded lattices $K$ and $L$
with a \zo homomorphism $\gf$ of $K$ into $L$, 
so that $\gy$ is represented as $\gf_{\Hom}$.
\end{abstract}

\maketitle

\section{Introduction}\label{S:Introduction}

In my paper \cite{gG13a}, I prove the 
Characterization Theorem of the Order of Principal Congruences:

\begin{theorem}\label{T:char}
Let $L$ be a bounded lattice and let $\Princl L$ denote 
the order of principal congruences of $L$. 
The order $\Princl L$ can be characterized as a bounded order. 
\end{theorem}
G. Cz\'edli~\cite{gC15} and~\cite{gC15a}   
extended this result to a bounded lattice and a \zo sub\-lattice.
Let $K$ be a \zo sublattice of a bounded lattice $L$.
Then the map
\begin{equation}\label{E:sub}
   \gy_{\Sub} \colon \conK{x, y} \mapsto \conL{x, y}
\end{equation}
is an isotone \zo map of $\Princl K$ into $\Princl L$. 
Observe that the \zo map $\gy_{\Sub}$ is
\emph{\zs}, that is, $\zero_K$ is the only element 
mapped by $\gy_{\Sub}$ to~$\zero_L$.

Now we can state G. Cz\'edli's result.

\begin{theorem}\label{T:Czedli}
Let $P$ and $Q$ be bounded orders and 
let $\gy$ be an isotone \zs \zo map from $P$ into $Q$. 
Then there exist a bounded lattice $L$, 
a \zo sublattice~$K$ of~$L$, so that $P$, $Q$, and $\gy$
are represented by $\Princl K$, $\Princl L$, and $\gy_{\Sub}$ 
up to isomorphism.
\end{theorem}

Theorem~\ref{T:char} follows from Theorem~\ref{T:Czedli} with
$P = Q$ and $\gy$ the identity map.
     
In this note we take up the analogous problem 
with homomorphic images rather than sublattices.
We start with the following observation.

\begin{lemma}\label{L:ontomap}
Let $K$ and $L$ be bounded lattices and let $\gf$ be a 
\zo homomorphism of $K$ into~$L$. 
Define 
\begin{equation}\label{E:princmap}
   \gy_{\Hom} \colon \conK{a,b} \mapsto \conL{\gf(a),\gf(b)}
\end{equation}
for $a \leq b \in K$. 
Then $\gf_{\Hom}$ is an isotone \zo map of $\Princl K$ into $\Princl L$.
\end{lemma}

Now we state our main result, the Representation Theorem for Order Triples.

\begin{theorem}\label{T:main}
Let $P$ and $Q$ be bounded orders and 
let $\gy$ be an isotone \zo map from $P$ into $Q$. 
Then there exist bounded lattices $K$, $L$, and
a \zo homomorphism $\gf$ of $K$ into~$L$, so that $P$, $Q$, and $\gy$
are represented by $\Princl K$, $\Princl L$, and $\gy_{\Hom}$, 
up to isomorphism.
\end{theorem}

We will consider \emph{lattice-triples}: $\E L = (K, L, \gf)$, where
$K$ and $L$ are bounded lattices and $\gf$ is 
a $\set{0, 1}$-homomorphism of $K$ into~$L$. 
Similarly, we consider \emph{order-triples} $\E P = (P, Q, \gy)$, where
$P$ and $Q$ are bounded orders and $\gy$ 
is an isotone $\set{0, 1}$-map of $P$ into~$Q$. 
By Lemma~\ref{L:ontomap}, a lattice-triple $\E L$ 
defines an order-triple $\E P$ in the natural way:
$P = \Princl K$, $Q = \Princl L$, and $\gy = \gy_{\tup{Hom}}$;
we shall use the notation $\Ordc(\E L)$ for this order-triple. 
A \emph{representable} order-triple~$\E P$ arises from a lattice-triple 
$\E L$ as $\Ordc(\E L)$.

Now we restate the Representation Theorem.

\begin{theorem}\label{T:repr}
Every order-triple is representable.
\end{theorem}

The proof of this theorem relies on the construction in \cite{gG13a} 
to prove Theorem~\ref{T:char}. 
To keep this paper short, we assume familiarity this construction. 
We also assume familiarity the basic concepts and notation of this field,
see any one of my books \cite{LTF}--\cite{LTS1}.

In Section~\ref{S:conghom}, we verify some elementary facts,
including that the map in \eqref{E:princmap} is well-defined. 
Section~\ref{S:Main} describes the main step in the proof of 
the Representation Theorem, proving it in a very special case.
Section~\ref{S:MainRepresentation} combines the result in
Section~\ref{S:Main} with Cz\'edli's Theorem~\ref{T:Czedli}
to verify the Representation Theorem.

We list some open problems in Section~\ref{S:Problems}.
In Appendix~A, we point out that a lattice construction of Cz\'edli's 
can be made smaller.

\section{Principal congruences and homomorphisms}\label{S:conghom}

We start by restating two well-known results, see for instance,
Lemma 229 and Theorem~230 in \cite{LTF}.

\begin{lemma}\label{L:cp}
Let L be a lattice, $a, b, c, d \in L$ with $a \leq b$ and $c \leq d$. Then $[a, b]$ is congruence-projective to $[c, d]$ 
if{}f there is an integer $m$ and there are elements 
$p_0, \ldots, p_{m - 1} \in L$ such that
\begin{align}
  t(a, p_0, \dots, p_{m - 1}) &= c,\label{E:pm}\\
  t(b, p_0, \dots, p_{m - 1}) &= d,\label{E:pm1}
\end{align}
where $t$ is defined by
\[
    t(x, y_0, \ldots , y_{m - 1}) = \cdots (((x \jj y_0) \mm y_1) \jj y_2) \mm
      \cdots.
\]
\end{lemma}
 
\begin{lemma}\label{L:cs}
Let $L$ be a lattice and let $a \leq b$ and $c \leq d$ in~$L$. Then
\begin{equation}\label{E:order}
\cngd c = d (\con{a, b})
\end{equation}
if{f}, for some ascending sequence
\begin{equation}\label{E:ascending sequence}
c = e_0 \leq e_1 \leq \cdots \leq e_n = d,
\end{equation}
the  congruence-projectivities
\begin{equation}\label{E:c-projectivities}
[a, b] \cproj \inv{e_j}{e_{j + 1}}
\end{equation}
hold for all $j = 0, \dots, n - 1$.
\end{lemma}

The next three lemmas are easy to prove.

\begin{lemma}\label{L:congproj}
Let $K$ and $L$ be lattices and let $\gf$ be a homomorphism of $K$ into~$L$. 
If~$a \leq b$, $x \leq y$, and $[a, b] \cproj [x, y]$ in~$K$,
then 
\begin{equation}\label{E:homcproj}
   [\gf(a), \gf(b)] \cproj [\gf(x), \gf(y)]
\end{equation}
holds in $L$.
\end{lemma}

\begin{proof}
By Lemma~\ref{L:cp}, if $[a, b] \cproj [x, y]$, 
then there is an integer $m$ and there are elements 
$p_0, \ldots, p_{m - 1} \in K$ such that
\eqref{E:pm} and \eqref{E:pm1} hold.
Since $\gf$ is a homomorphism, we get that 
\begin{align}
   t(\gf(a), \gf(p_0), \dots, \gf(p_{m - 1})) &= \gf(x),
   \label{E:pm3}\\
    t(\gf(b), \gf(p_0), \dots, \gf(p_{m - 1})) &= \gf(y).
   \label{E:pm4}
\end{align}
Again, by Lemma~\ref{L:cs}, \eqref{E:pm3} and \eqref{E:pm4}
imply that \eqref{E:homcproj} holds.
\end{proof}

\begin{lemma}\label{L:congspred}
Let $K$ and $L$ be lattices and let $\gf$ be a homomorphism of $K$ into~$L$. 
If~$a \leq b$, $x \leq y$, and $\cng x = y (\con{a,b})$ in~$K$,
then 
\begin{equation}\label{E:homcproj1}
   \cng {\gf(x)} = {\gf(y)} (\con{ \gf(a), \gf(b)})
\end{equation}
holds in $L$.
\end{lemma}

\begin{proof}
By Lemma~\ref{L:cs}, 
there is a sequence $c = e_0 \leq e_1 \leq \cdots \leq e_n = d$,
such that \eqref{E:c-projectivities} holds.
By Lemma~\ref{L:congproj}, 
\begin{equation}\label{E:homcproj2}
[\gf(a), \gf(b)] \cproj [\gf(e_{j}), \gf(e_{j+1})]
\end{equation}
for all $j = 0, \dots, n - 1$.
By Lemma~\ref{L:cs}, \eqref{E:homcproj1} holds.
\end{proof}

We rewrite Lemma~\ref{L:congspred} as follows.

Let $K$ and $L$ be lattices, let $\gf$ be a homomorphism of $K$ into~$L$,
and let us assume that $a,b,c,d \in K$. 
Then
\begin{equation}\label{E:isotone}
\con{a,b} \geq \con{c,d} \text{ in $K$ implies that } 
\con{\gf(a),\gf(b)} \geq \con{\gf(c),\gf(d)} \text{ in $L$}.
\end{equation}

Let us call the lattice-triple $\E L = (K, L, \gf)$ \emph{surjective},
if $\gf$ maps $K$ \emph{onto} $L$.
Similarly, an order-triple $\E P = (P, Q, \gy)$ is \emph{surjective},
if $\gy$ maps $P$ \emph{onto} $Q$. 
An~order-triple $\E P = (P, Q, \gy)$ has a \emph{surjective representation}
if there is a surjective lattice-triple representing it.

\begin{lemma}\label{L:onto}
If the order-triple $\E P = (P, Q, \gy)$ 
has a surjective representation $\E L = (K, L, \gf)$,
then $\E P$ is surjective.
\end{lemma}

\begin{proof}
Let $u,v \in L$. Since $\gf$ is surjective, 
there are elements $a, b \in K$ with $\gf(a) = u$ and $\gf(b) = v$.
It follows from \eqref{E:princmap} that 
$\gf_{\Hom}(\con{a,b}) = \con{u,v}$, 
so the map $\gf_{\Hom} = \gy$ is surjective.
\end{proof}

Now we prove Lemma~\ref{L:ontomap}.
Since $K$ and $L$ are bounded lattices and $\gf$ is a 
\zo homo\-morphism of $K$ into~$L$, 
it follows that $\gf_{\tup{Hom}}$ is a \zo map.

By Lemma~\ref{L:onto}, the map $\gf_{\Hom}$ is surjective.
Applying \eqref{E:isotone} twice to $\con{a,b} = \con{c,d}$, 
we conclude that $\con{\gf(a),\gf(b)} = \con{\gf(c),\gf(d)}$ in $L$, 
proving that $\gf_{\Hom}$ is a map.
\eqref{E:isotone} also verifies that $\gf_{\Hom}$ is isotone,
concluding the proof of  Lemma~\ref{L:ontomap}.

\section{The main step}\label{S:Main}

The main step in the proof of the Representation Theorem 
is its verification in a very special case.

We need some notation. 
For an order-triple $\E P = (P, Q, \gy)$, we define 
\[
    \Top \E P = \setm{x \in P}{\gy(x) > 0_Q} \uu \set{0_P}
\]
and let $\Top \gy$ be the restriction of $\gy$ to $\Top \E P$.
We also need the ``bottom'' of $\E P$: 
\[
  \Btm \E P = \setm{x \in P}{\gy(x) = 0_Q}.  
\]
Note that
\begin{align*}
   \Top \E P \uu \Btm \E P &= P,\\
   \Top \E P \ii \Btm \E P &= \set{0_P}.
\end{align*}
Finally, for a bounded order $P$, let $\Lat P$ be the lattice
we construct in \cite{gG13a} as an extension of 
$\Frame P$ (see Figure~\ref{F:F} with $X = \es$)
by inserting the lattice $G(p,q)$, see Figure~\ref{F:G(p,q)},
as a sublattice into $\Frame P$ for all $p, q \in P$ satisfying
$0_P < p < q < 1_P$.

\begin{figure}[t!]
\centerline{\includegraphics[scale=0.7]{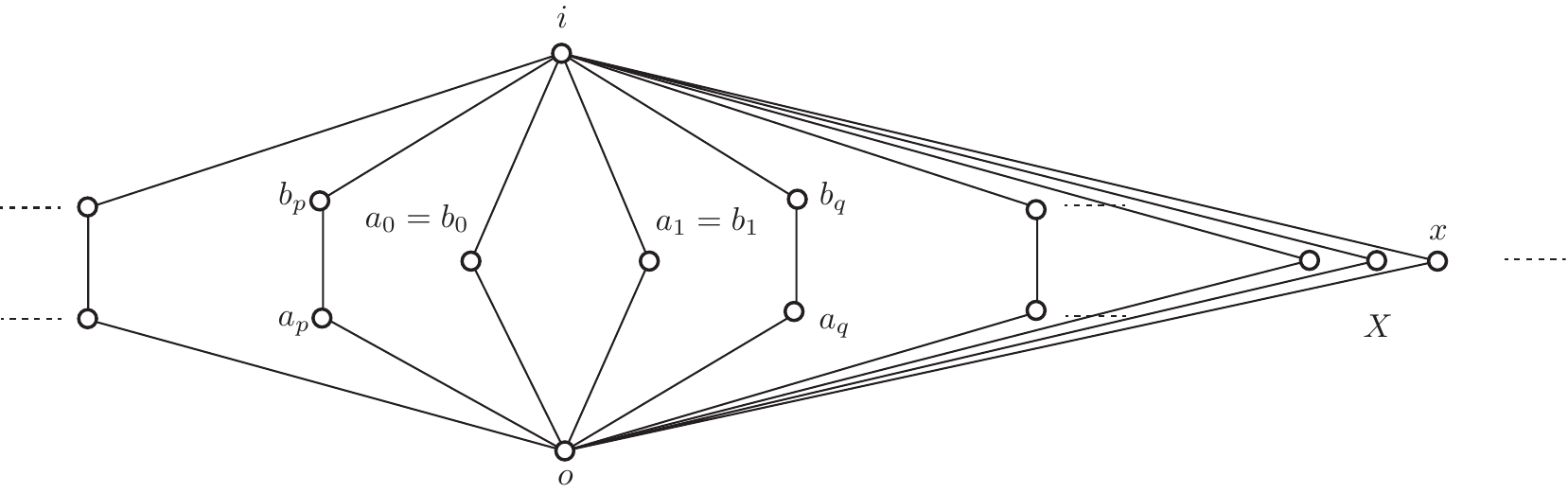}}
\caption{The lattice $\Frame_X P$}\label{F:F}
\end{figure}

\begin{figure}[t!]
\centerline{\includegraphics{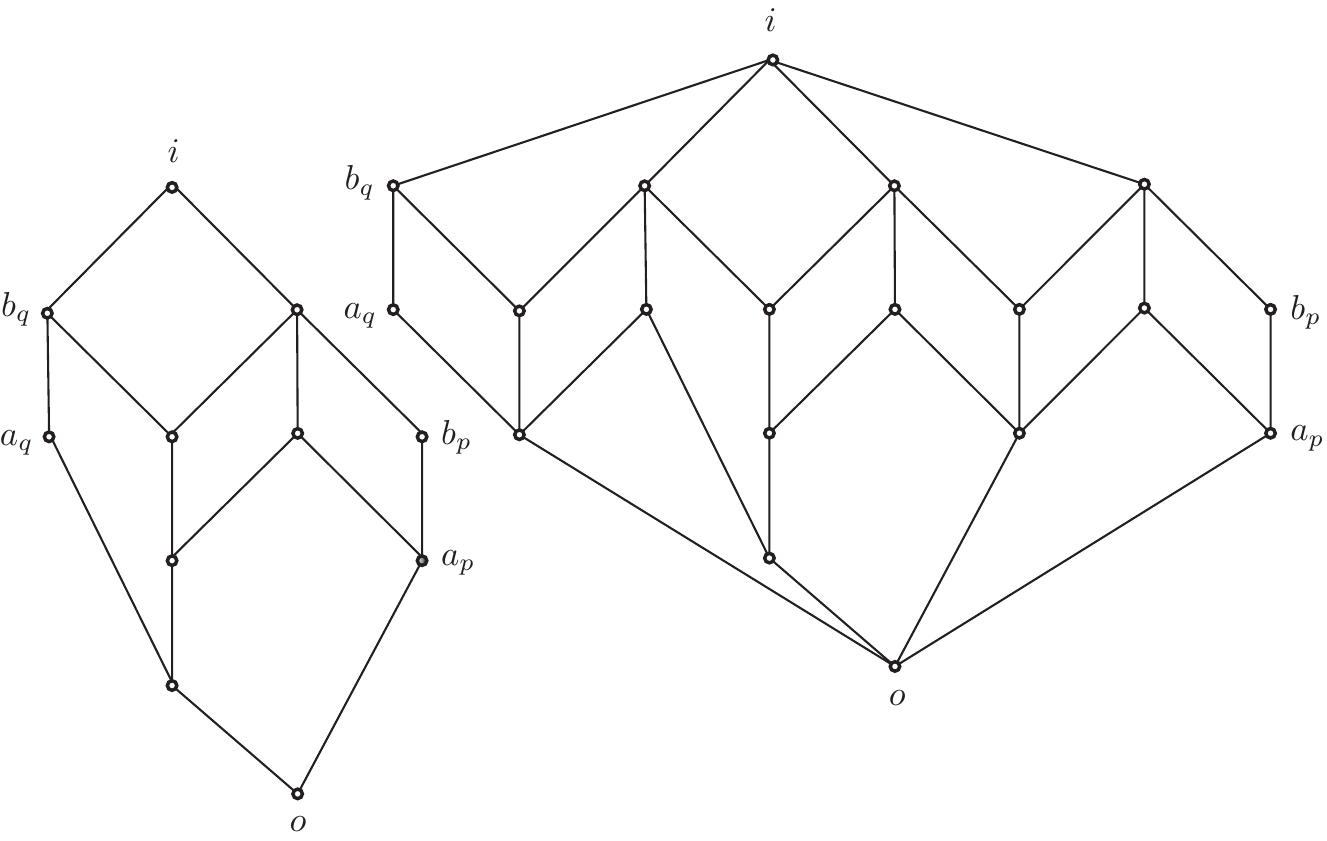}}
\caption{The lattices $G(p,q)$ and $G(p,q)^\tup{Ext}$}\label{F:G(p,q)}
\end{figure}

\begin{lemma}\label{L:surj}
Let $\E P = (P, Q, \gy)$ be a surjective the order-triple.
If $\Top \gy$ is an isomorphism between $\Top \E P$ and~$Q$,
then $\E P$ has a surjective representation
$\E L = (K, L, \gf)$, where $K = \Lat P$, $L = \Lat_X Q$
with $X = \Btm \E P$. 
\end{lemma}

\begin{proof}
Let $\E P = (P, Q, \gy)$ be a surjective order-triple
and let $\Top \gy$ be an isomorphism between $\Top \E P$ and $Q$.

A \emph{\zo isolating congruence} $\bga$ of~a bounded lattice $L$, 
is a nontrivial congruence~$\bga$, such that $\set{0}$ and $\set{1}$ are 
congruence blocks of~\bga. 
With a \zo isolating congruence $\bga$ of $\Lat P$,
we associate a subset of the order $P^- = P - \set{0,1}$:
\[
   \Base\bga = \setm{p \in P^-}{\cngd a_p=b_p(\bga)}.
\]

We now restate some parts of Lemmas 6--9 of \cite{gG13a}.

\medskip

\emph{The correspondence 
$
   \gd \colon \bga \to \Base\bga
$
is an order preserving bijection between
the order of \zo isolating congruences of\, $\Lat P$ and $\Down P^-$, 
the order of down sets of~$P^-$.
Let $(\Down P^-)^\tup{t}$ be the order $\Down P^-$ 
with a new unit element, $P$, added.
We extend \gd by $\one \to P$.
Then \gf is an isomorphism between $\Con (\Lat P)$ and $(\Down P^-)^\tup{t}$. 
The maps \gd and $\gd^{-1}$ both preserve the property of being principal.}
\medskip

It follows that $P \iso \Princl (\Lat P)$.

Now we describe a lattice-triple $\E L = (K, L, \gf)$.
We define $K = \Lat P$. 
Since $\Btm \E P$ is a down-set, it follows from the above statement that
there is a congruence~$\bga$ of $K$ with $\Btm \E P = \Base\bga$.
Define the bounded lattice $L = K /\bga$, 
and let $\gf$ be the natural \zo-homomorphism of $K$ onto $L$.
So $\E L = (K, L, \gf)$ is a lattice-triple.
By the definitions of $K$, $L$, and $\gf$,
we have that $\Ordc(\E L) = \E P$, in fact, 
the lattice-triple $\E L$ is a surjective representation 
of the order-triple~$\E P$. 

An element $u$ in a bounded lattice $A$ is a \emph{universal complement}
if $u$ is complementary to every other element of $A -\set{0,1}$.
Now consider the lattice $\Frame_X P$, see Figure~\ref{F:F},
which is the same as $\Frame P$
except that we add the set $X$ such that each $x \in X$
is a universal complement;
$\Frame P$ is the special case $X = \es$.
We then construct the lattice $\Lat_X P$ 
the same way as we constructed $\Lat P$ but starting with $\Frame_X P$;
the elements $x \in X$ remain universal complements in the larger lattice.
Note that the lattice $K$ can be represented in the form $\Lat_X Q$,
where $X = \Btm \E P$, concluding the proof of the lemma. 
\end{proof}

There are a several papers with lattice constructs based on $\Frame P$
(see my paper \cite{gG13a} and G. Cz\'edli's papers
\cite{gC14}--\cite{gC15a}---with more to come).
We have just discussed the construction in \cite{gG13a},
concluding that adding a set $X$ of universal complements to~$\Frame P$
allows the same lattice construction 
and the order of principal congruences 
in the lattice constructed remains the same.

\section{Proving the Representation Theorem}\label{S:MainRepresentation}

Let $\E P = (P, Q, \gy)$ be a an order-triple. 
We consider the bounded order $R = \Top \E P$ and define
the isotone map $\ga \colon P \to R$ as follows:
\begin{equation}\label{E:alpha}
   \ga(x) = 
            \begin{cases}
               x,   &\text{ for }x \in \Top \E P;\\
               0_P=0_R, &\text{ otherwise.}
            \end{cases}
\end{equation}

\begin{figure}[t!]
\centerline{\includegraphics{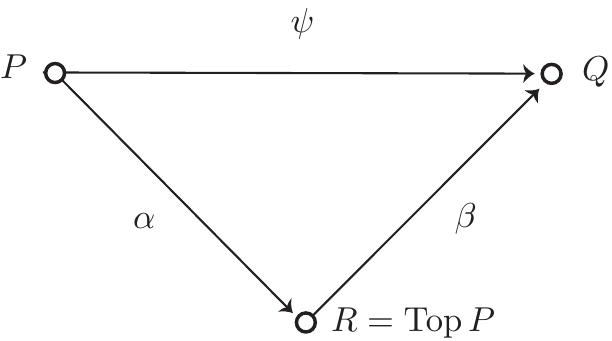}}
\caption{The order-triples $\E P_\ga = (P, R, \ga)$ and $\E P_\gb = (R, Q, \gb)$}\label{F:commdiag}
\end{figure}

We also define the bounded isotone map $\gb \colon R \to Q$ as 
the restriction of $\gy$ to $R = \Top \E P$,
see Figure~\ref{F:commdiag}. Note that $\gb\ga = \gy$.

This defines the order-triples
$\E P_\ga = (P, R, \ga)$ and $\E P_\gb = (R, Q, \gb)$.

Since $\E P_\ga$ is a surjective order-triple
and $\ga$ is an isomorphism between $\Top \E P$ and $R$,
it follows from Lemma~\ref{L:surj} that $\E P_\ga$
has a surjective representation $\E L_\ga = (K, M, \gf_{\bga})$ 
with $K = \Lat P$ and $M = \Lat_X (\Top \E P)$, where $X = \Btm \E P$.

Now we make a small but important technical change. 
In the construction of $\Lat P$ we replace $G(p,q)$ 
with the extended version $G(p,q)^\tup{Ext}$ of G. Cz\'edli~\cite{gC14}, 
obtaining the lattice $\Lat^\tup{Ext} P$, $\Lat P$ extended.
Similarly in $M$, obtaining $M^\tup{Ext}$.

Now we apply Cz\'edli's Theorem~\ref{T:Czedli}.
In $\E P_\gb$, the map $\gb$ is \zs, therefore, we can apply 
Theorem~\ref{T:Czedli} to $R$, $Q$, and~$\gb$. 
So we have obtained
\begin{enumeratea}
\item a lattice-triple, $\E L_\ga = (K, M, \ga)$ 
satisfying  $\Ordc(\E L_\ga) = \E P_\ga$;
\item a lattice-triple $(M, L, \gf_{\Sub})$
satisfying $\Ordc(\E L_\gb) = \E P_\gb$.
\end{enumeratea}

We conclude that $\E L = (K, L, \gb\ga)$ is a  lattice-triple,  
and $\Ordc(\E L) = \E P$, verifying the Representation Theorem.
\begin{figure}[hbt]
\centerline{\includegraphics{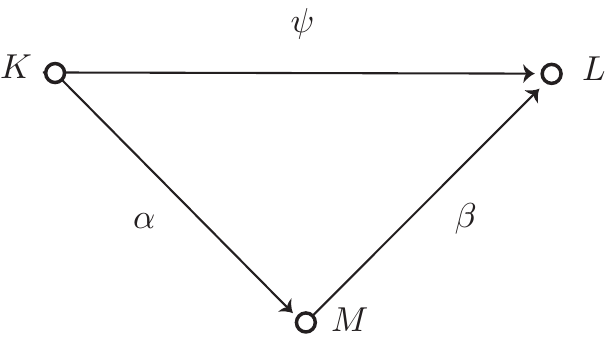}}
\caption{The lattice-triples $(K, M, \ga)$ and $(M, L, \gb)$}\label{F:commdiaglat}
\end{figure}

\section{Problems}\label{S:Problems}

Theorem~\ref{T:char} was extended in G. Cz\'edli~\cite{gC14}
to countable lattices and countable orders with zero.

\begin{problem}\label{P:xx}
Can we extended Theorem~\ref{T:repr} to the countable case?
\end{problem}

This would seem to be technically rather difficult.

\begin{problem}\label{P:xx}
For a finite semimodular lattice $L$, characterize $\Princl L$.
\end{problem}

\begin{problem}\label{P:xx}
For a finite planar semimodular lattice $L$, characterize $\Princl L$.
\end{problem}

\appendix

\section{A smaller construction}\label{S:smaller}

Cz\'edli's proof of Theorem~\ref{T:Czedli} is easy to outline 
(but not so easy to compute).
Let $\E P = (P, Q,\gy)$ be an order triple.
We form the bounded order $R = P \uu Q$, a disjoint union with $0_P, 0_Q$ and $1_P, 1_Q$ identified.
So $R$ is a bounded order containing $P$ and $Q$ as \zo suborders. 
Therefore, we have $\Frame P$ as a \zo suborder of $\Frame R$.

For $p < q$ in $P$ and for $p < q$ in $R$, 
we insert $G(p, q)^\tup{Ext}$ into $\Frame R$
so that $\con{a_p, b_p} < \con{a_q, b_q}$ will hold.
Then we need to ensure that 
\begin{equation}\label{E:wish}
\con{a_p, b_p} = \con{a_{\gy(p)}, b_{\gy(q)}}.
\end{equation}
We accomplish this by inserting $G(p, \gy(p))^\tup{Ext}$ 
and $G(\gy(p),p)^\tup{Ext}$. 
The first insertion gives us
\[
   \con{a_p, b_p} \leq \con{a_{\gy(p)}, b_{\gy(q)}},
\]
while the second gives us
\[
   \con{a_p, b_p} \geq \con{a_{\gy(p)}, b_{\gy(q)}}.
\]
The two together yield \eqref{E:wish}.
\begin{figure}[b!]
\centerline{\includegraphics{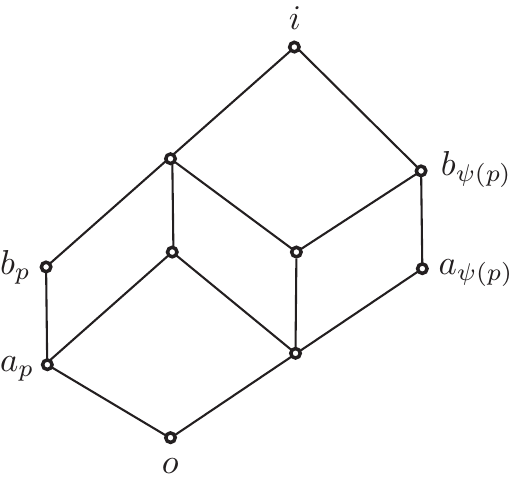}}
\caption{Adding $4$ elements, $\Equi(p)$}\label{F:Equi}
\end{figure}

This step requires that we insert $30$ elements.
Now we reduce this number to~$4$. 
Consider the lattice $\Equi(p)$ of Figure~\ref{F:Equi}.
For all $p \in P$, we add the four new elements of $\Equi(p)$
to the construction to obtain an $8$ element sublattice. 
This forces that $\con{a_p, b_p} = \con{a_{\gy(p)}, b_{\gy(q)}}$.
We do not present the verification of this construction.

By also replacing the sublattices $G(p, q)^\tup{Ext}$ by $G(p, q)$,
if $P^-$ has $n_P$ elements and $c_P$ (nonreflexive) compatibilities
while $Q^-$ has $c_Q$ (nonreflexive) compatibilities,
then Cz\'edli's construction adds $15c_P + 15c_Q + 30n_P$ elements.
The construction in this section adds $7c_P + 7c_Q + 4n_P$ elements.

Note that in Cz\'edli's constructions, in general, 
it is important to use $G(p, q)^\tup{Ext}$ rather than $G(p, q)$.

\end{document}